\documentclass[10pt, a4paper]{article}
\usepackage{amsmath, amsfonts}
\usepackage[usenames]{color}
\usepackage[latin1]{inputenc}
\usepackage[table, dvipsnames]{xcolor}
\usepackage{array}
\newcolumntype{P}[1]{>{\centering\arraybackslash}p{#1}}
\newcolumntype{M}[1]{>{\centering\arraybackslash}m{#1}}
\usepackage{multirow}
\usepackage{mathrsfs}
\usepackage{hyperref} 
\hypersetup{bookmarks = true, 
                  colorlinks = true, 
                  linkcolor = blue, 
                  citecolor = blue, 
                  }
\newcommand{\xdownarrow}[1]{%
  {\left\downarrow\vbox to #1{}\right.\kern-\nulldelimiterspace}
}

\newcommand{\xuparrow}[1]{%
  {\left\uparrow\vbox to #1{}\right.\kern-\nulldelimiterspace}
}

\usepackage{graphicx}
\usepackage{amssymb}
\usepackage{amsmath}

\newtheorem{theorem}{Theorem}[section]

\newtheorem{definition}[theorem]{Definition}
\newtheorem{example}[theorem]{Example}

\newtheorem{lemma}[theorem]{Lemma}

\newtheorem{remark}[theorem]{Remark}

\newenvironment{proof}[1][Proof]{\noindent\textbf{#1.} }{\ \rule{0.5em}{0.5em}}

\title{\textbf{A formula to calculate the invariant $J$ of a quasi-homogeneous map germ}}
\author{ \ \ \ \\{O.N. Silva\footnote{O.N. Silva: Instituto de Matemáticas, Universidad Nacional Autónoma de México (UNAM), Cuernavaca, México. \hspace{5cm} e-mail: otoniel@im.unam.mx}}}
\date{}

\begin{document}

\maketitle

\begin{abstract}
In this work, we consider a quasi-homogeneous, corank $1$, finitely determined map germ $f$ from $(\mathbb{C}^2,0)$ to $(\mathbb{C}^3,0)$. We consider the invariants $m(f(D(f))$ and $J$, where $m(f(D(f))$ denotes the multiplicity of the image of the double point curve $D(f)$ of $f$ and $J$ denotes the number of tacnodes that appears in a stabilization of the transversal slice curve of $f(\mathbb{C}^2)$. We present formulas to calculate $m(f(D(f))$ and $J$ in terms of the weights and degrees of $f$.
\end{abstract}

\section{Introduction}

$ \ \ \ \ $ In this work, we consider a quasi-homogeneous, corank $1$, finitely determined map germ $f$ from $(\mathbb{C}^2,0)$ to $(\mathbb{C}^3,0)$. Local coordinates can be chosen so that these map germs can be written in the form $f(x,y)=(x, \tilde{p}(x,y),\tilde{q}(x,y))$, for some quasi-homogeneous function germs $\tilde{p}, \tilde{q} \in m_2^2$, where $m_2$ is the maximal ideal of the local ring of holomorphic function germs in two variables $\mathcal{O}_2$ (see Lemma \ref{lemma corank 1}).

In \cite{slice}, Nuño-Ballesteros and Marar studied the transversal slice of a corank $1$ map germ $f:(\mathbb{C}^2,0)\rightarrow(\mathbb{C}^3,0)$ (see \cite[Section 3]{slice}). They show in some sense that if a set of generic conditions are satisfied, then the transverse slice curve contains information on the geometry of $f$. They also introduced the invariants $C$, $T$ and $J$, which are defined as the number of cusps, triple points and tacnodes that appears in a stabilization of the transversal slice of $f$, respectively. They showed in \cite[Prop. 3.10]{slice} that the numbers $C$ and $T$ are the same as the numbers of cross-caps and triple points that appear in a stabilization of $f$, which are usually denoted by $C(f)$ and $T(f)$, respectively, and were defined by Mond in \cite{mond6}. On the other hand, the invariant $J$ is related to both, the delta invariant of the transverse slice curve of $f(\mathbb{C}^2)$ and $N(f)$ (see \cite[page 1388]{slice}), where $N(f)$ is the Mond's invariant defined in \cite{mond6}.

In \cite{mondformulas}, Mond presented formulas to calculate the invariants $C(f)$, $T(f)$ and $\mu(D(f))$ of a quasi-homogeneous finitely determined map germ of any corank in terms of the weights and degrees of $f$, where $\mu(D(f))$ denotes the Milnor number of the double point curve $D(f)$ of $f$ (see Section \ref{double point sec} for the definition of $D(f)$ and Th. \ref{mondformulas}). So, a natural question is:\\

\noindent \textbf{Question 1:} Let $f:(\mathbb{C}^2,0)\rightarrow (\mathbb{C}^3,0)$ be a quasi-homogeneous, corank $1$, finitely determined map germ. Can the invariant $J$ be calculated in terms of the weights and the degrees of $f$?\\ 

It follows by Propositions 3.5 and 3.10 and Corollary 4.4 of \cite{slice} that  

\begin{center}
$J=\dfrac{1}{2}\bigg(\mu(D(f))-C(f)-1\bigg)-3T(f)+m(f(D(f))$,
\end{center}

\noindent where $m(f(D(f))$ denotes the multiplicity of the image of the double point curve $D(f)$. So, using Mond's formulas for $C(f)$, $T(f)$ and $\mu(f(D(f))$ (Th. \ref{mondformulas}) we conclude that Question $1$ has a positive answer if and only if there is a formula to calculate the invariant $m(f(D(f))$ in terms of the weights and the degrees of $f$. So, another natural question is:\\

\noindent \textbf{Question 2:} Let $f$ be as in Question $1$. Can the invariant $m(f(D(f))$ be calculated in terms of the weights and degrees of $f$?\\ 

In \cite[Proposition 6.2]{otoniel1}, Ruas and the author provided answers to both questions above in the case where $f$ is homogeneous. In this work, using a normal form for $f$ (Lemma \ref{lemma corank 1}), we present a positive answer for both questions without any restriction on the weights and degrees of $f$. More precisely, we present in Theorem \ref{main theo} formulas to calculate both invariants, $m(f(D(f))$ and $J$, in terms of the weights and the degrees of $f$.  We finish this work by calculating the invariants $m(f(D(f))$ and $J$ in some examples to illustrate our formulas.

\section{Preliminaries}

$ \ \ \ \ $ Throughout this paper, given a finite map $f:\mathbb{C}^2\rightarrow \mathbb{C}^3$, $(x,y)$ and $(X,Y,Z)$ are used to denote systems of coordinates in $\mathbb{C}^2$ (source) and $\mathbb{C}^3$ (target), respectively. Also, $\mathbb{C} \lbrace x_1,\cdots,x_n \rbrace \simeq \mathcal{O}_n$ denotes the local ring of convergent power series in $n$ variables. The letters $U,V$ and $W$ are used to denote open neighborhoods of $0$ in $\mathbb{C}^2$, $\mathbb{C}^3$ and $\mathbb{C}$, respectively. We also use the standard notation of singularity theory as the reader can find in Wall's survey paper \cite{wall}.

\subsection{Double point curves for corank 1 map germs}\label{double point sec}

$ \ \ \ \ $ In this section, we deal only with of corank $1$ maps from $\mathbb{C}^2$ to $\mathbb{C}^3$. For the general definition of double point spaces, see for instance \cite[Section 1]{ref7}, \cite[Section 2]{ref9} and \cite[Section 3]{mond6}. 

Consider a finite and holomorphic map $f: U\rightarrow \mathbb{C}^3 $, where $U$ is an open neighbourhood of $0$ in $\mathbb{C}^2$. The double point space of $f$, denoted by $D(f)$, is defined (as a set) by

 \begin{center}
$D(f):=\lbrace (x,y) \in U \ : \ f^{-1}(f(x,y))\neq \lbrace (x,y)\rbrace \rbrace \cup \Sigma(f)$,
 \end{center}

\noindent where $\Sigma(f)$ is the ramification set of $f$. We also consider the lifting of the $D(f)$ in $U \times U$, denoted by $D^2(f)$, given by the pairs $((x,y),(x^{'},y^{'}))$ such that either $f(x,y)=f(x^{'},y^{'})$ with $(x,y) \neq (x^{'},y^{'})$ or $(x,y)=(x^{'},y^{'})$ with $(x,y) \in \Sigma(f)$.

We need to choose convenient analytic structures for the double point space $D(f)$ and the lifting of the double point space $D^2(f)$. As we said in Introduction, when $f$ has corank $1$, local coordinates can be chosen so that these map germs can be written in the form $f(x,y)=(x, \tilde{p}(x,y),\tilde{q}(x,y))$, for some function germs $\tilde{p}, \tilde{q} \in m_2^2$, where $m_2$ is the maximal ideal of $\mathcal{O}_2$. In this case, we define \textit{the lifting of the double point space} $D^2(f)$, (as a complex space) by

\begin{center}
$D^2(f)=V \displaystyle \left( x-x^{'},\dfrac{\tilde{p}(x,y)-\tilde{p}(x,y^{'})}{y-y^{'}}, \dfrac{\tilde{q}(x,y)-\tilde{q}(x,y^{'})}{y-y^{'}} \right)$
\end{center}

\noindent where $(x,y,x^{'},y^{'})$ are coordinates of $\mathbb{C}^2 \times \mathbb{C}^2$.\\

Once the lifting $D^2(f) \subset U \times U$ is defined as a complex analytic space, we now consider its image $D(f)$ (also as a complex analytic space) on $U$ by the projection 

\begin{center}
$\pi: U \times U \rightarrow U$ 
\end{center}

\noindent onto the first factor, which will be considered with the structure given by Fitting ideals. We also consider the double point space in the target, that is, the image of $D(f)$ by $f$, denoted by $f(D(f))$, which will also be consider with the structure given by Fitting ideals. 

We remark that given a finite morphism of complex spaces $h:X\rightarrow Y$ the push-forward $h_{\ast}\mathcal{O}_{X}$ is a coherent sheaf of $\mathcal{O}_{Y}-$modules (see \cite[Chapter 1]{grauert}) and to it we can (as in \rm\cite[Section 1]{ref13}) associate the Fitting ideal sheaves $\mathcal{F}_{k}(h_{\ast}\mathcal{O}_{X})$. Notice that the support of $\mathcal{F}_{0}(h_{\ast}\mathcal{O}_{X})$ is just the image $h(X)$. Analogously, if $h:(X,x)\rightarrow(Y,y)$ is a finite map germ then we denote by $ F_{k}(h_{\ast}\mathcal{O}_{X})$ the \textit{k}th Fitting ideal of $\mathcal{O}_{X,x}$ as $\mathcal{O}_{Y,y}-$module. In this way, we have the following definition.

\begin{definition} Let $f:U \rightarrow V$ be a finite mapping, where $U$ and $V$ are open neighbourhoods of $0$ in $\mathbb{C}^2$ and $\mathbb{C}^3$, respectively.\\

\noindent (a) Let ${\pi}|_{D^2(f)}:D^2(f) \subset U \times U \rightarrow U$ be the restriction to $D^2(f)$ of the projection $\pi$. The \textit{double point space of $f$} is the complex space

\begin{center}
$D(f)=V(\mathcal{F}_{0}({\pi}_{\ast}\mathcal{O}_{D^2(f)}))$.
\end{center}

\noindent Set theoretically we have the equality $D(f)=\pi(D^{2}(f))$.\\

\noindent (b) The \textit{double point space of $f$ in the target} is the complex space $f(D(f))=V(\mathcal{F}_{1}(f_{\ast}\mathcal{O}_2))$. Notice that the underlying set of $f(D(f))$ is the image of $D(f)$ by $f$.\\ 

\noindent (c) Given a finite map germ $f:(\mathbb{C}^{2},0)\rightarrow (\mathbb{C}^3,0)$, \textit{the germ of the double point space of $f$} is the germ of complex space $D(f)=V(F_{0}(\pi_{\ast}\mathcal{O}_{D^2(f)}))$. \textit{The germ of the double point space of $f$ in the target} is the germ of the complex space $f(D(f))=V(F_{1}(f_{\ast}\mathcal{O}_2))$. 
\end{definition}

\begin{remark} If $f:U \subset \mathbb{C}^2 \rightarrow V \subset \mathbb{C}^3 $ is finite and generically $1$-to-$1$, then $D^2(f)$ is Cohen-Macaulay and has dimension $1$ (see \rm\cite[\textit{Prop. 2.1}]{ref9}\textit{). Hence, $D^2(f)$, $D(f)$ and $f(D(f))$ are complex analytic curves. In this case, without any confusion, we also call these complex spaces by the ``lifting of the double point curve'', the ``double point curve'' and the ``image of the double point curve'', respectively.}
\end{remark}

\subsection{Finite determinacy and the invariants $C(f)$ and $T(f)$}

\begin{definition}\label{def a equi}(a) Two map germs $f,g:(\mathbb{C}^2,0)\rightarrow (\mathbb{C}^3,0)$ are $\mathcal{A}$-equivalent, denoted by $g\sim_{\mathcal{A}}f$, if there exist map germs of diffeomorphisms $\eta:(\mathbb{C}^2,0)\rightarrow (\mathbb{C}^2,0)$ and $\xi:(\mathbb{C}^3,0)\rightarrow (\mathbb{C}^3,0)$, such that $g=\xi \circ f \circ \eta$.\\

\noindent (b) A map germ $f:(\mathbb{C}^2,0) \rightarrow (\mathbb{C}^3,0)$ is finitely determined (or $\mathcal{A}$-finitely determined) if there exists a positive integer $k$ such that for any $g$ with $k$-jets satisfying $j^kg(0)=j^kf(0)$ we have $g \sim_{\mathcal{A}}f$. 

\end{definition}

Consider a finite map germ $f:(\mathbb{C}^2,0)\rightarrow (\mathbb{C}^3,0)$. By Mather-Gaffney criterion (\rm\cite[Theorem 2.1]{wall}), $f$ is finitely determined if and only if there is a finite representative $f:U \rightarrow V$, where $U\subset \mathbb{C}^2$, $V \subset \mathbb{C}^3$ are open neighbourhoods of the origin, such that $f^{-1}(0)=\lbrace 0 \rbrace$ and the restriction $f:U \setminus \lbrace 0 \rbrace \rightarrow V \setminus \lbrace 0 \rbrace$ is stable.

This means that the only singularities of $f$ on $U \setminus \lbrace 0 \rbrace$ are cross-caps (or Whitney umbrellas), transverse double and triple points. By shrinking $U$ if necessary, we can assume that there are no cross-caps nor triple points in $U$. Then, since we are in the nice dimensions of Mather (\rm\cite[p. 208]{mather}), we can take a stabilization of $f$, 

\begin{center}
$F:U \times D \rightarrow \mathbb{C}^4$, $F(z,s)=(f_{s}(z),s)$, 
\end{center}

\noindent where $D$ is a neighbourhood of $0$ in $\mathbb{C}$. 

\begin{definition} We define

\begin{center}
$T(f):= \sharp $ of triple points of $f_s$ $ \ \ \ $ and  $ \ \ \ $ $C(f):= \sharp $ of cross-caps of $f_s$,
\end{center}

\noindent where $s\neq 0$.

\end{definition}

It is well known that the numbers $T(f)$ and $C(f)$ are independent of the particular choice of the stabilization and they are also analytic invariants of $f$ (see for instance \rm\cite{mond7}).

We remark that the space $D(f)$ plays a fundamental role in the study of the finite determinacy. In \cite[Theorem 2.14]{ref7}, Marar and Mond presented necessary and sufficient conditions for a map germ $f:(\mathbb{C}^n,0)\rightarrow (\mathbb{C}^p,0)$ with corank $1$ to be finitely determined in terms of the dimensions of $D^2(f)$ and other multiple points spaces. When $(n,p)=(2,3)$, in \cite{ref9}, Marar, Nu\~{n}o-Ballesteros and Pe\~{n}afort-Sanchis extended this criterion of finite determinacy to the corank $2$ case. More precisely, they proved the following result.

\begin{theorem}\rm(\cite[\textit{Corollary 3.5}]{ref9})\label{criterio} \textit{
Let $f:(\mathbb{C}^2,0)\rightarrow(\mathbb{C}^{3},0)$ be a finite and generically $1$-to-$1$ map germ. Then $f$ is finitely determined if and only if $\mu(D(f))$ is finite (equivalently, $D(f)$ is a reduced curve).}
\end{theorem}

\subsection{Identification and Fold components of $D(f)$}

$ \ \ \ \ $ When $f:(\mathbb{C}^2,0)\rightarrow (\mathbb{C}^3,0)$ is finitely determined, the restriction of a representative of $f$ to $D(f)$ is finite. In this case, $f_{|D(f)}$ is generically $2$-to-$1$ (i.e; $2$-to-$1$ except at $0$). On the other hand, the restriction of $f$ to an irreducible component $D(f)^i$ of $D(f)$ can be generically $1$-to-$1$ or $2$-to-$1$. This motivates us to give the following definition which is from \cite[Def. 4.1]{otoniel3} (see also \cite[Def. 2.4]{otoniel1} and \cite[Sec. 3]{otoniel4}).

\begin{definition}\label{typesofcomp} Let $f:(\mathbb{C}^2,0)\rightarrow (\mathbb{C}^3,0)$ be a finitely determined map germ. Let $f:U\rightarrow V$ be a representative, where $U$ and $V$ are neighbourhoods of $0$ in $\mathbb{C}^2$ and $\mathbb{C}^3$, respectively. Consider an irreducible component $D(f)^j$ of $D(f)$.\\

\noindent (a) If the restriction ${f_|}_{D(f)^j}:D(f)^j\rightarrow V$ is generically $1$-to-$1$, we say that $D(f)^j$ is an \textit{identification component of} $D(f)$.\\ 

In this case, there exists an irreducible component $D(f)^i$ of $D(f)$, with $i \neq j$, such that $f(D(f)^j)=f(D(f)^i)$. We say that $D(f)^i$ is the \textit{associated identification component to} $D(f)^j$ or that the pair $(D(f)^j, D(f)^i)$ is a \textit{pair of identification components of} $D(f)$.  \\

\noindent (b) If the restriction ${f_|}_{D(f)^j}:D(f)^j\rightarrow V$ is generically $2$-to-$1$, we say that $D(f)^j$ is a \textit{fold component of} $D(f)$.\\

\noindent (c) We define the sets $IC(D(f))=\lbrace $identification components of $D(f) \rbrace $ and  $FC(D(f))= \lbrace  $fold components of $D(f) \rbrace $. And we define the numbers $r_i(D(f)):=\sharp IC(D(f))$ and $r_f(D(f)):= \sharp FC(D(f))$.  
\end{definition}

\begin{remark} Let $f$ and $g$ be finitely determined map germs from $(\mathbb{C}^2,0)$ to $(\mathbb{C}^3,0)$. Suppose that $g\sim_{\mathcal{A}}f$ and write $g=\xi \circ f \circ \eta$ as in Definiton \rm\ref{def a equi}. \textit{Consider representatives $f,g:U\rightarrow V$ of $f$ and $g$. Let $D(g)^i$ be an irreducible component of $D(g)$ and consider its corresponding image by $\eta$, $D(f)^i:=\eta(D(g)^i)$, which is an irreducible component of $D(f)$. Note that , ${f_|}_{D(f)^j}:D(f)^i\rightarrow V$ is generically $k$-to-$1$ if and only if ${g_|}_{D(g)^i}:D(g)^i\rightarrow V$ is generically $k$-to-$1$, where $k=1,2$. Hence}

\begin{center}
$r_i(D(f))=r_i(D(g))$ $ \ \ \ $ \textit{and} $ \ \ \ $ $r_f(D(f))=r_f(D(g))$.
\end{center}

\end{remark}

The following example illustrates the two types of irreducible components of $D(f)$ presented in Definition \ref{typesofcomp}.

\begin{example}\label{example}
\textit{Let $f(x,y)=(x,y^2,xy^3-x^7y)$ be the singularity $C_7$ of Mond's list} \rm(\cite[\textit{p.378}]{mond6}). \textit{In this case, $D(f)=V(xy^2-x^7)$. Then $D(f)$ has three irreducible components given by}

\begin{center}
$D(f)^1=V(y-x^3), \ \ \ $ $D(f)^2=V(y+x^3) \ \ $ and $ \ \ D(f)^3=V(x)$. 
\end{center}

\textit{Notice that $(D(f)^1$, $D(f)^2)$ is a pair of identification components and $D(f)^3$ is a fold component. Hence, we have that $r_i(D(f))=2$ and $r_f(D(f))=1$. We have also that $f(D(f)^3)=V(X,Z)$ and $f(D(f)^1)=f(D(f)^2)=V(Y-X^6,Z)$ (see Figure \rm\ref{figura1}).}

\begin{figure}[h]
\centering
\includegraphics[scale=0.3]{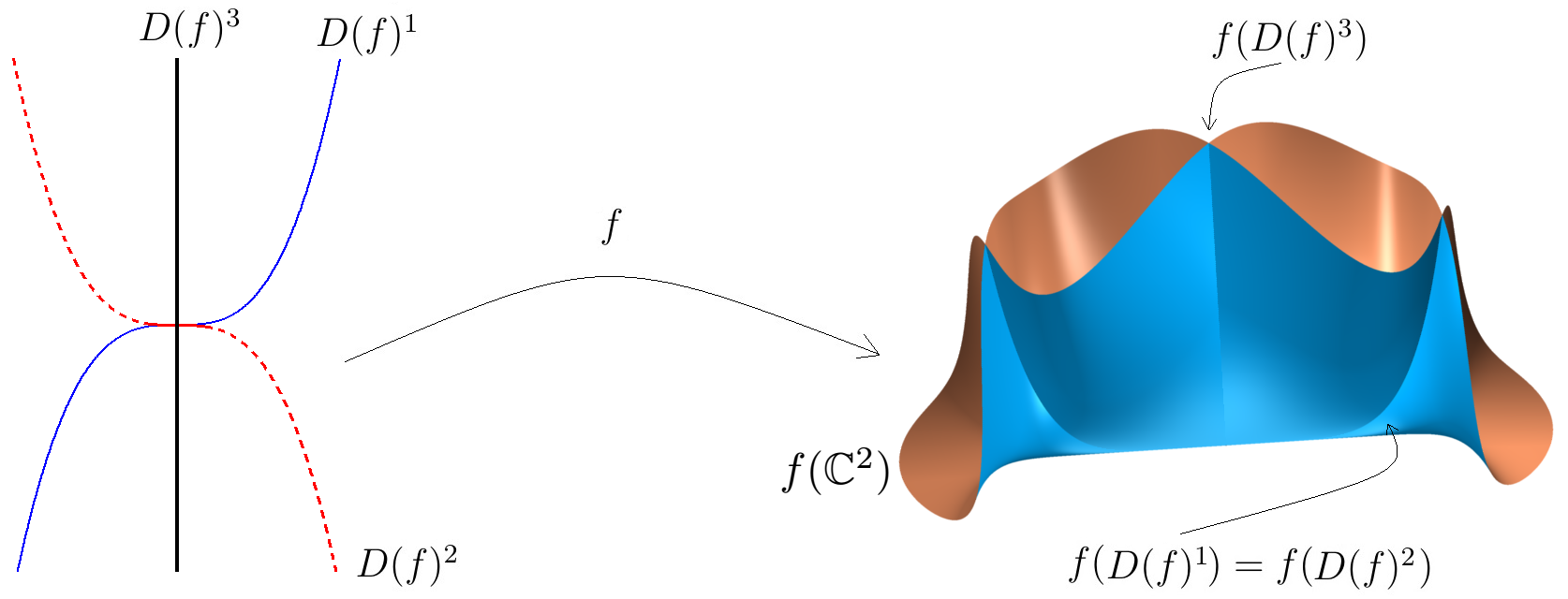} 
\caption{Identification and fold components of $D(f)$ (real points)}\label{figura1}
\end{figure}

\end{example}

\begin{remark} In the Example \rm\ref{example}\textit{, we have made use of the software Surfer }\rm\cite{surfer}\textit{.} 
\end{remark}

\subsection{Quasi-homogeneous map germs and Mond's formulas}\label{quasi-hom sec}

\begin{definition} A polynomial $p(x_1,\cdots,x_n)$ is \textit{quasi-homogeneous} if there are positive integers $w_1,\cdots,w_n$, with no common factor and an integer $d$ such that $p(k^{w_1}x_1,\cdots,k^{w_n}x_x)=k^dp(x_1,\cdots,x_n)$. The number $w_i$ is called the weight of the variable $x_i$ and $d$ is called the weighted degree of $p$. We also write $w(p)$ to denote the weighted degree of $p$. In this case, we say $p$ is of type $(d; w_1,\cdots,w_n)$.
\end{definition}

This definition extends to polynomial map germs $f:(\mathbb{C}^n,0)\rightarrow (\mathbb{C}^p,0)$ by just requiring each coordinate function $f_i$ to be quasi-homogeneous of type $(d_i; w_1,\cdots,w_n)$. In particular, when $f:(\mathbb{C}^2,0)\rightarrow (\mathbb{C}^3,0)$ is quasi-homogeneous, we say that $f$ is quasi-homogeneous of type $(d_1,d_2,d_3; a,b)$, where here we change the classical notation $w_1,w_2$ of the weights of $x$ and $y$ by $a,b$, for simplicity.

The following lemma describes a normal form for a class of quasi-homogeneous map germs from $(\mathbb{C}^2,0)$ to $(\mathbb{C}^3,0)$ which will be considered in this work.

\begin{lemma}(Normal form lemma)\label{lemma corank 1} Let $g(x,y)=(g_1(x,y),g_2(x,y),g_3(x,y))$ be a quasi-homogeneous, corank $1$, finitely determined map germ of type $(d_1,d_2,d_3; a,b)$. Then $g$ is $\mathcal{A}$-equivalent to a quasi-homogeneous map germ $f$ with type $(d_{i_1}=a,d_{i_2},d_{i_3};a,b)$, which is written as

\begin{equation}\label{eq12}
f(x,y)=(x, y^n+xp(x,y), \alpha y^m+ xq(x,y)),
\end{equation}

\noindent for some integers $n,m\geq 2$, $\alpha \in \mathbb{C}$, $p,q \in \mathcal{O}_2$, $p(x,0)=q(x,0)=0$, where $(d_{i_1},d_{i_2},d_{i_3})$ is a permutation of $(d_1,d_2,d_3)$ such that $d_{i_2}\leq d_{i_3}$.
\end{lemma}

\begin{proof} Since $g$ has corank $1$, $g_i$ is a regular in $x$ or $y$ for some $i$. Without lose of generality, suppose that $g_1$ is regular in $x$. Thus $g_1(x,y)=\gamma x + g^{'}_1(x,y)$, where $\gamma \in \mathbb{C}$, $\gamma \neq 0$. We have that $g$ is quasi-homogeneous, this implies that $g^{'}_1(x,y)=\theta y^a$, where $\theta \in \mathbb{C}$. Also, if $\theta\neq 0$, then $b=1$. Consider the analytic isomorphisms $\eta: (\mathbb{C}^2,0)\rightarrow (\mathbb{C}^2,0)$ and $\xi:(\mathbb{C}^3,0)\rightarrow (\mathbb{C}^3,0)$ defined by 

\begin{center}
$\eta(x,y)=(x-(\gamma^{-1}\theta) y^a,y)$ $ \ \ \ $ and $ \ \ \ $ $\xi(X,Y,Z)=(\gamma^{-1}X,Y,Z)$.
\end{center}

Note that the map $\tilde{g}:=\xi \circ g \circ \eta$ is a quasi-homogeneous map germ of type $(d_{1},d_{2},d_{3};a,b)$. There exist integers $v_1,v_2$, complex numbers $\alpha_1,\alpha_2$ and polynomials $p_1,p_2 \in \mathcal{O}_2$ such that $\tilde{g}$ is written as
 
\begin{center}
$\tilde{g}(x,y)=(x,\alpha_1 y^{v_1} + xp_1(x,y), \alpha_2 y^{v_2}+xp_2(x,y))$,
\end{center}

\noindent with $2\leq v_1,v_2$. After a change of coordinates (which does not change the quasi-homogeneous type of $\tilde{g}$), we can assume that $p_1(x,0)=p_2(x,0)=0$. So, we need to show that $\alpha_i \neq 0$ for some $i$ with $w(\alpha_iy^{v_i}+xp_i(x,y))\leq w(\alpha_jy^{v_j}+xp_j(x,y))$, where $i\neq j$ and $i,j \in \lbrace 1,2\rbrace$. Since $\tilde{g}$ is also finitely determined, in particular, it is finite, so either $\alpha_1 \neq 0$ or $\alpha_2 \neq 0$. Thus we have three cases.\\

\noindent (1) $\alpha_1, \alpha_2 \neq 0$.\\
\noindent (2) $\alpha_1 =0$ and $\alpha_2 \neq 0$.\\ 
\noindent (3) $\alpha_1 \neq 0$ and $\alpha_2 =0$.\\

(1) Suppose that $\alpha_1,\alpha_2 \neq 0$. In this case, we can suppose that $\alpha_1=\alpha_2=1$ (applying the change of coordinates $(X,Y,Z)\mapsto (X,\alpha_1^{-1}Y,\alpha_2^{-1}Z)$). Define $n:=min \lbrace v_1,v_2 \rbrace$ and $p:=p_i$ if $n=v_i$. Define also $m:=v_1$ and $q:=p_1$ if $n=v_2$ or $m:=v_2$ and $q:=p_2$ if $n=v_1$. So, $\tilde{g}$ is $\mathcal{A}$-equivalent to $f(x,y)=(x,y^n+xp(x,y),y^m+xq(x,y))$. Note that $w(y^n+xp(x,y)\leq w(y^m+xq(x,y))$.\\  

(2) If $\alpha_1 = 0$ and $\alpha_2 \neq 0$, then the restriction of $\tilde{g}$ to $V(x)$ is $v_2$-to-$1$. Since $\tilde{g}$ if finitely determined and singular, we have that $v_2=2$. In this case, $\tilde{g}(x,y)=(x,xp_1(x,y),\alpha_2 y^2+xp_2(x,y))$. Again, after a change of coordinates, we can assume that $\alpha_2=1$. Write $x p_1(x,y)= \lambda_1 x^{k_1}y^{s_1}+\cdots+\lambda_l x^{k_l}y^{s_l}$.\\ 

\noindent \textbf{Statement:} We have that $w(y^2+xp_2(x,y)) \leq w(xp_1(x,y))$.\\

\textit{Proof of Statement:} Note that $p_1(x,y)\not\equiv 0$, otherwise the set of cross-caps $C(\tilde{g})$ of $\tilde{g}$ is not finite and hence $\tilde{g}$ is not finitely determined. Since $p_1(x,0)=0$, we have that $s_i\geq 1$ for all $i$. If $s_i\geq 2$ for some $i$, then the statement is clear and after a change of coordinates, we see that $\tilde{g}$ is $\mathcal{A}$-equivalent to

\begin{center}
$f(x,y)=(x,y^2+xp(x,y),xq(x,y))$,
\end{center}

\noindent where $p=p_2$, $q=p_1$ and $w(y^2+xp(x,y))\leq w(xq(x,y))$, as desired.

Now, suppose that $s_i=1$ for all $i$. In this case, after a change of coordinates, we can write $\tilde{g}$ as $\tilde{g}(x,y)=(x,x^{k}y,y^2+xp_2(x,y))$, for some $k\geq 1$. We have that $D^2(\tilde{g})=V(x-x^{'}, \ x^{k} \ , y+y^{'}+x(p_2(x,y)-p_2(x,y^{'}))/(y-y^{'}))$ which is not reduced if $k\geq 2$. Since $\tilde{g}$ is finitely determined, by Theorem \ref{criterio} and \cite[Theorem 2.4]{ref9} we have that $k=1$. In this case, $D^2(\tilde{g})=V(x,x^{'},y+y^{'})\subset \mathbb{C}^2 \times \mathbb{C}^2$ is a smooth curve and $\tilde{g}$ does not have any triple points. It follows by \cite[Th. 2.14]{ref7} that $\tilde{g}$ is stable. Hence, $\tilde{g}$ is $\mathcal{A}$-equivalent to $f(x,y)=(x,y^2,xy)$, which is considered with quasi-homogeneous type $(1,2,2;1,1)$ and has the desired properties, that is, $w(y^2)\leq w(xy)$.\\

Now, the analysis of case (3) is analogous.\end{proof}

\begin{remark}
Some versions of Lemma \rm\ref{lemma corank 1} \textit{are well know by specialists (see for instance} \rm\cite[\textit{Lemma 4.1}]{mond7}\textit{). We include its proof for completeness. Given a quasi-homogeneous, corank $1$, finitely determined map germ, we will assume in the proofs throughout this paper that $f$ is written in the normal form in} \rm(\ref{eq12}), \textit{presented in Normal form lemma}.
\end{remark}

In \cite{mondformulas}, Mond showed that if $f$ is quasi-homogeneous then the invariants $C(f)$, $T(f)$ and also $\mu(D(f))$ are determined by the weights and the degrees of $f$. More precisely, he showed the following result.

\begin{theorem}\label{mondformulas} \rm({\cite{mondformulas}}) \textit{Let $f:(\mathbb{C}^2,0)\rightarrow (\mathbb{C}^3,0)$ be a quasi-homogeneous finitely determined map germ of type $(d_1,d_2,d_3;a,b)$. Then}

\begin{center}
$C(f)=\dfrac{1}{ab}\bigg((d_2-a)(d_3-b)+(d_1-b)(d_3-b)+(d_1-a)(d_2-a)\bigg)$,
\end{center}

\begin{center}
$T(f)=\dfrac{1}{6ab}(\delta-\epsilon)(\delta-2\epsilon)+\dfrac{C(f)}{3} \ \ \ \ \ $ and $ \ \ \ \ \ \mu(D(f))=\dfrac{1}{ab}(\delta-\epsilon-a)(\delta-\epsilon-b)$.
\end{center}

\noindent \textit{where $\epsilon = d_{1}+d_{2}+d_{3}-a-b$ and $\delta=d_{1}d_{2}d_{3}/(ab)$.}

\end{theorem}

\section{Formulas for the invariants $J$ and $m(f(D(f)))$}

$ \ \ \ \ $ We note that by a parametrization of an irreducible complex germ of curve $(X,0) \subset (\mathbb{C}^n,0)$ we mean a primitive parametrization, that is, a holomorphic and generically $1$-to-$1$ map germ $n$ from $(\mathbb{C},0)$ to $(\mathbb{C}^n,0)$, such that $n(W,0)\subset (X,0)$ (see for instance \cite[Section 3.1]{greuel16}). Before we present our main result, we will need the following lemma.

\begin{lemma}\label{lemma aux4} Let $f$ be a finitely determined, corank $1$, quasi-homogeneous map germ of type $(d_1,d_2,d_3;a,b)$. Write $f$ as in Lemma \ref{lemma corank 1}, that is, $f(x,y)=(x, y^n+xp(x,y), \alpha y^m+ xq(x,y))$, with $d_2\leq d_3$. Then\\ 

\noindent (a) If $V(y)$ is an irreducible component of $D(f)$, then $V(y)\in IC(D(f))$ and $a=1$.\\

\noindent (b) $D(f)=V(\lambda(x,y))$, where $\lambda(x,y)$ is a quasi-homogeneous polynomial of type $\left(\dfrac{d_2d_3}{b}-d_2-d_3+b;a,b\right)$ and

\begin{equation}\label{eq3}
\lambda(x,y)=\displaystyle { x^{s}\prod_{i=1}^{r}}(y^a-\alpha_i x^b),
\end{equation}

\noindent where $\alpha_i \in \mathbb{C}$ are all distinct, $r=\dfrac{(d_2-b)(d_3-b)-sab}{ab^2}\geq 0$ and either $s =0$ or $s=1$.\\

\noindent (c) If $a>d_2$, then $p(x,y)=0$. That is, $f(x,y)=(x,y^n,\alpha y^m+xq(x,y))$.\\

\noindent (d) If $s=1$ in \rm(\ref{eq3}), \textit{that is, if $V(x)$ is an irreducible component of $D(f)$, then it is a fold component of $D(f)$.} \\

\noindent \textit{(e) If $\alpha=0$, then $n=2$.}\\

\noindent \textit{(f) $s=0$ if and only if $\alpha \neq 0 $ and $gcd(n,m)=1$. In other words, $s=1$ if and only if either $\alpha =0$, or $\alpha \neq 0$ and $gcd(n,m)=2$.}

\end{lemma}

\begin{proof} Consider a representative $f:U\rightarrow V$ of $f$.\\

((a) and (b)) Suppose that $V(y)$ is an irreducible component of $D(f)$. Consider the parametrization of $V(y)$ given by the map $\varphi_0:W\rightarrow U$, defined by $\varphi_0(u)=(u,0)$. So, $f\circ \varphi_0:W\rightarrow V$, defined as

\begin{equation}\label{eq11}
(f\circ \varphi_0)(u)=(u,0,0),
\end{equation}
 
\noindent is a parametrization of $f(V(y))$. Since $f\circ \varphi_0$ is $1$-to-$1$, $V(y)$ is an identification component of $D(f)$. Since $f$ is quasi-homogeneous and finitely determined, we have that $\lambda(x,y)$ is a quasi-homogeneous polynomial of type $\left(\dfrac{d_2d_3}{b}-d_2-d_3+b;a,b\right)$, by \cite[Prop. 1.15]{mondformulas}. 

The only irreducible quasi-homogeneous polynomials with $w(x)=a$ and $w(y)=b$ in the ring of polynomials $\mathbb{C}[x,y]$ are $x,y$ and $y^a-\alpha_i x^b$, with $\alpha_i \in \mathbb{C}$ and $\alpha_i \neq 0$. Since the ring of polynomials $\mathbb{C}[x,y]$ is an unique factorization domain, each irreducible factor of $\lambda$ is on the form of $x,y$ or $y^a-\alpha_i x^b$. By Theorem \ref{criterio}, $\lambda$ is reduced, hence the irreducible factors of $\lambda$ are all distinct. So, $\lambda$ can take the following form:

\begin{equation}\label{eq8}
\lambda(x,y)=\displaystyle { x^{s}y^{l}\prod_{i=1}^{r^{'}}}(y^a-\alpha_i x^b),
\end{equation}

\noindent where $s,l \in \lbrace 0,1 \rbrace$, $r^{'}\geq 0$, $\alpha_i$ are all distinct and $\alpha_i \neq 0$ for all $i$. We note that if $r^{'}=0$, then $\displaystyle { \prod_{i=1}^{0}}(y^a-\alpha_i x^b)=1$ (the empty product).

Consider the parametrization of $V(y^a-\alpha_i x^b)$ given by the map $\varphi_{\alpha_i}:W\rightarrow U$, defined by $\varphi_{\alpha_i}(u)=(u^a,\gamma_i u^b)$, where $\gamma_i=\alpha_i^{1/a}$. So, $f\circ \varphi_{\alpha_i}:W\rightarrow V$, defined as

\begin{equation}\label{eq10}
(f\circ \varphi_{\alpha_i})(u)=(u^a,\gamma_{1,i}u^{d_2},\gamma_{2,i}u^{d_3}), 
\end{equation}

\noindent is a parametrization of $f(V(y^a-\alpha_i x^b))$, for some $\gamma_{1,i},\gamma_{2,i} \in \mathbb{C}$. 

Since $f(V(x))\cap f(V(y))=\lbrace (0,0,0) \rbrace$, the associated identification component of $V(y)$ is a curve $V(y^a-\alpha_j   x^b)$ for some $\alpha_j \neq 0$. Since, $(V(y),V(y^a-\alpha_j x^b))$ is a pair of identification components of $D(f)$, comparing (\ref{eq11}) and (\ref{eq10}) we see that $\gamma_{1,j}=\gamma_{2,j}=0$ and $a=1$. Consequently, the expression (\ref{eq8}) can be rewritten as follows:

\begin{center}
$\lambda(x,y)=\displaystyle { x^{s}\prod_{i=1}^{r}}(y^a-\alpha_i x^b)$,
\end{center}

\noindent where $s \in \lbrace 0,1 \rbrace$, $r=\dfrac{(d_2-b)(d_3-b)-sab}{ab^2}\geq 0$, $\alpha_i$ are all distinct and we allow one of the $\alpha_i's$ to be zero.\\

(c) Suppose that $p(x,y)\neq 0$. By assumption of the normal form of $f$, we have that $p(x,0)=0$, hence $p(x,y)$ is not a constant. If $a>d_2=bn$, then $w(xp(x,y))>a>d_2$,  a contradiction.\\

(d) Note that for all $i$, $V(x)$ and $V(y^a-\alpha_i x^b)$ have distinct images. Hence, $V(x)\notin IC(D(f))$. So if $V(x) \subset D(f)$, then we conclude that $V(x) \in FC(D(f))$.\\ 

(e) Suppose that $\alpha =0$, then restriction of $f$ to $V(x)$ is $n$-to-$1$. Since $f$ is finitely determined, either $n=1$ or $n=2$. Since $f$ has corank $1$, we conclude that $n=2$.\\ 

(f) By (d), we have that $V(x)$ is an irreducible component of $D(f)$ if and only if $f\circ \varphi$ is generically $2$-to-$1$, where $\varphi: W\rightarrow U$ is the parametrization of $V(x)$, defined by $\varphi(u)=(0,u)$. Suppose that $V(x) \subset D(f)$. This implies that either $\alpha=0$ or $\alpha \neq 0$ and $gcd(n,m)=2$. On the other hand, if $\alpha =0$ then by (e) we have that $n=2$ and $f\circ \varphi$ is generically $2$-to-$1$. Hence, in this case $V(x) \subset D(f)$. If $\alpha \neq 0$ and $gcd(n,m)=2$, then again $f\circ \varphi$ is generically $2$-to-$1$ and hence $V(x) \in FC(D(f))$. 

Now, suppose that $V(x)$ is not an irreducible component of $D(f)$. Since $f$ is generically $1$-to-$1$, by (d) we have that $V(x)\not\subset D(f)$ if and only if $f\circ \varphi$ is generically $1$-to-$1$ if and only if $\alpha \neq 0$ and $gcd(n,m)=1$.\end{proof}\\

In the following result, $m(f(D(f)))$ denotes the Hilbert-Samuel multiplicity of the maximal ideal of the local ring $\mathcal{O}_{f(D(f))}$ of $(f(D(f)),0)$ (or equivalently, the multiplicity of $f(D(f))$ at $0$). Also, $J$ denotes the number of tacnodes that appears in a stabilization of the transversal slice curve of $f(\mathbb{C}^2)$ (see \cite[Def. 3.7]{slice}).

Note that if $\varphi:W\subset \mathbb{C}\rightarrow V \subset \mathbb{C}^3$, $\varphi(u)=(u^m,\varphi_2(u),\varphi_3(u))$ is a Puiseux parametrization of a reduced curve in $\mathbb{C}^3$, then its multiplicity is $m$ (see for instance \cite[page 98]{chirka}). We remark that given a germ of reduced curve $(C,0)\subset (\mathbb{C}^n,0)$, it is not true that its irreducible components are also reduced, see for instance \cite[Example 4.12]{otoniel6} where $(X_0,0)=(X_0^1 \cup X_0^2,0)$ is a germ of reduced curve in $(\mathbb{C}^3,0)$, but $(X_0^1,0)$ is not reduced at $0$.

Suppose that $f:(\mathbb{C}^2,0)\rightarrow (\mathbb{C}^3,0)$ is finitely determined. So, $D(f)$ is a reduced curve, by Theorem \ref{criterio}. It follows by \cite[Th. 4.3]{ref9} that $f(D(f))$ is also a reduced curve. However, given an irreducible component $f(D(f)^i)$ of $f(D(f))$, it may contain a (embedded) zero dimensional component, and therefore may not be reduced. If this is the case, we say that $f(D(f)^i)$ is a generically reduced curve. Recently, the author and Snoussi showed in \cite[Lemma 4.8]{otoniel5} that if $(C,0)$ is a germ of generically reduced curve and $(|C|,0)$ is its associated reduced curve, then the multiplicities of $(C,0)$ and $(|C|,0)$ at $0$ are equal. Hence, we also can calculate the multiplicity of $f(D(f)^i)$ considering its reduced structure and using a corresponding Puiseux parametrization for it. We are now able to present our main result.

\begin{theorem}\label{main theo} Let $f$ be a finitely determined, corank $1$, quasi-homogeneous map germ. Write $f$ as in Lemma \ref{lemma corank 1}, that is, $f(x,y)=(x, y^n+xp(x,y), \alpha y^m+ xq(x,y))$ and it is of type $(d_1=a,d_2,d_3;a,b)$ such that $d_2\leq d_3$. Then

\begin{center}
$m(f(D(f))=\dfrac{1}{2ab^2}\bigg((d_2-b)(d_3-b)c+s a b (d_2-c) \bigg)$ $ \ \ \ \ $ and\\

$J=\dfrac{1}{2ab^2}\bigg((d_2-b)(d_3-b)(c-3b)+b(\delta-\epsilon-a)(\delta-\epsilon-b)+b(\epsilon-\delta)(\delta-2\epsilon)+sab(d_2-c)-ab^2\bigg)$
 \end{center} 

\noindent where $\epsilon = d_{2}+d_{3}-b$, $\delta=d_{2}d_{3}/b$, $c=min\lbrace a,d_2 \rbrace$ and 

\[ s =   \left\{
\begin{array}{ll}
      0 & if \ \alpha \neq 0 \ and \ gcd(n,m)=1, \\
      1 & otherwise.     
\end{array} 
\right. \]

\end{theorem}

\begin{proof} Take a representative $f:U\rightarrow V$ of $f$. By Lemma \ref{lemma aux4} (b), we have that $D(f)=V(\lambda(x,y))$, where

\begin{center}
$\lambda(x,y)=\displaystyle { x^{s}\prod_{i=1}^{r}}(y^a-\alpha_i x^b)$,
\end{center}

\noindent $s=0$ or $1$, $\alpha_i \in \mathbb{C}$ are all distinct and $r=\dfrac{(d_2-b)(d_3-b)-sab}{ab^2}$. 

Set $\mathscr{C}_{\alpha_i}:=V(y^a-\alpha_i x^b)$. As in the proof of Lemma \ref{lemma aux4}, consider a parametrization $\varphi_{\alpha_i}: W \rightarrow U $ of  $\mathscr{C}_{\alpha_i}$ defined by $\varphi_{\alpha_i}(u)=(u^a,\gamma_i u^b)$, where $W$ is an open neighbourhood of $0$ in $\mathbb{C}$ and $\gamma_i:=\alpha_i^{1/a}$. So, if $\mathscr{C}_{\alpha_i}$ is an identification component of $D(f)$, then the mapping $\tilde{\varphi}_{\alpha_i}:= f\circ \varphi_{\alpha_i}: W \rightarrow V$, defined by 

\begin{equation}\label{eq1}
\tilde{\varphi}_{\alpha_i}:=(u^a,\gamma_{1,i} u^{d_2}, \gamma_{2,i} u^{d_3}),
\end{equation}

\noindent is a parametrization of $f(\mathscr{C}_{\alpha_i})$, for some $\gamma_{1,i}, \gamma_{2,i} \in \mathbb{C}$. On the other hand, if $\mathscr{C}_{\alpha_i}$ is a fold component of $D(f)$, then the mapping $\varphi^{'}_{\alpha_i}: W \rightarrow V$, defined by 

\begin{equation}\label{eq2}
\varphi^{'}_{\alpha_i}(u):=(u^{a/2},\gamma_{1,i}^{'} u^{{d_2}/2}, \gamma_{2,i}^{'} u^{{d_3}/2}),
\end{equation}

\noindent is a parametrization of $f(\mathscr{C}_{\alpha_i})$, for some $\gamma_{1,i}^{'}, \gamma_{2,i}^{'} \in \mathbb{C}$. Set $c:=min\lbrace a, d_2 \rbrace$. Note that if $a>d_2$, then $\gamma_{1,i},\gamma_{1,i}^{'}\neq 0$, by Lemma \ref{lemma aux4} (c). It follows by (\ref{eq1}) and (\ref{eq2}) that

\[ m(f(\mathscr{C}_{\alpha_i})) =   \left\{
\begin{array}{ll}
      c & if \ \mathscr{C}_{\alpha_i} \in IC(D(f)),  \\
      c/2 & if \ \mathscr{C}_{\alpha_i} \in FC(D(f)).  \\
  \end{array} 
\right. \]

Set $\mathscr{C}:=V(x)$. If $\mathscr{C} \subset D(f)$, then by Lemma \ref{lemma aux4} (d) we have that it is a fold component of $D(f)$. In this case, the map $\varphi :W\rightarrow V$ defined by

\begin{equation}\label{eq4}
\varphi(u)=(0,u^{n/2},\alpha u^{m/2})
\end{equation}

\noindent is a parametrization of $\mathscr{C}$. It follows by (\ref{eq4}) that $m(\mathscr{C})=n/2$. Hence, we have that

\begin{center}
$m(f(D(f)))=\bigg(\dfrac{r_i(D(f))}{2}\bigg)c+(r_f(D(f))-s)\bigg(\dfrac{c}{2}\bigg)+s \bigg(\dfrac{n}{2}\bigg)$.
\end{center}

By Lemma \ref{lemma aux4} (b), we have that $r_i(D(f))+r_f(D(f))-s=r=\dfrac{(d_2-b)(d_3-b)-sab}{ab^2}$. Also, note that $n=d_2/b$. Hence,

\begin{equation}\label{eq6}
m(f(D(f)))=\bigg(\dfrac{(d_2-b)(d_3-b)-sab}{ab^2}\bigg)\bigg(\dfrac{c}{2}\bigg)+s \bigg(\dfrac{d_2}{2b}\bigg)=\dfrac{1}{2ab^2}\bigg((d_2-b)(d_3-b)c+s a b (d_2-c) \bigg).
\end{equation}

It follows by Propositions 3.5 and 3.10 and Corollary 4.4 of \cite{slice} that  

\begin{equation}\label{eq7}
J=\dfrac{1}{2}\bigg(\mu(D(f))-C(f)-1\bigg)-3T(f)+m(f(D(f)).
\end{equation}

Now, by Theorem \ref{mondformulas} and the expressions (\ref{eq6}) and (\ref{eq7}), we have that

\begin{center}
$J=\dfrac{1}{2ab}\bigg((\delta-\epsilon-a)(\delta-\epsilon-b)+(\epsilon-\delta)(\delta-2\epsilon)-3(d_2-b)(d_3-b)-ab\bigg)+\dfrac{1}{2ab^2}\bigg((d_2-b)(d_3-b)c+s a b (d_2-c) \bigg)$\\

$=\dfrac{1}{2ab^2}\bigg((d_2-b)(d_3-b)(c-3b)+b(\delta-\epsilon-a)(\delta-\epsilon-b)+b(\epsilon-\delta)(\delta-2\epsilon)+sab(d_2-c)-ab^2\bigg)$,

\end{center}

\noindent where $\epsilon = d_{2}+d_{3}-b$ and $\delta=d_{2}d_{3}/b$.\end{proof}

\section{Examples}

$ \ \ \ \ $ When we look to the formulas in Theorem \ref{main theo}, we identify four cases. More precisely, we identified four situations depending on the values that $c$ and $s$ assume. In this section, we present examples illustrating these situations.

\begin{example}\label{example2} (a) ($c=a$ and $s=0$) Consider the $F_4$-singularity of Mond's list \rm\cite{mond6}\textit{, given by}

\begin{center}
 $f(x,y)=(x,y^2,y^5+x^3y)$. 
\end{center}
 
\textit{We have that $f$ is quasi-homogeneous of type $(4,6,15; 4,3)$. In this case $c=4$ and $s=0$. By Theorem} \rm\ref{main theo} \textit{we have that $m(f(D(f)))=2$ and $J=3$.}\\

\noindent \textit{(b) ($c=a$ and $s=1$) Consider the map germ}

\begin{center}
$f(x,y)=(x,y^4,y^6+x^5y-5x^3y^3+4xy^5)$,
\end{center}

\noindent \textit{which is quasi-homogeneous of type $(1,4,6; 1,1)$. In this case $c=1$ and $s=1$. Again by Theorem} \rm \ref{main theo} \textit{we have that $m(f(D(f)))=9$ and $J=39$. We remark that $f$ is from} \rm\cite[\textit{Example 5.5}]{otoniel1}\textit{, where finite determinacy is proved.}\\

\noindent \textit{(c) ($c=d_2$ and $s=0$) Consider the $H_2$-singularity of of Mond's list, given by}

\begin{center}
$f(x,y)=(x,y^3,y^5+xy)$,
\end{center}

\noindent \textit{which is quasi-homogeneous of type $(4,3,5; 4,1)$. Using Theorem} \rm\ref{main theo} \textit{we have that $m(f(D(f)))=3$ and $J=2$.}\\

\noindent \textit{(d) ($c=d_2$ and $s=1$) Consider the map germ}

\begin{center}
$f(x,y)=(x,y^2,x^2y-xy^5)$,
\end{center}  

\noindent \textit{which is quasi-homogeneous of type $(4,2,9; 4,1)$. We have that $D(f)=V(x(x-y^4))$ which is reduced. So, by Theorem} \rm\ref{criterio} \textit{we have that $f$ is finitely determined. Using Theorem} \rm\ref{main theo} \textit{we have that $m(f(D(f)))=2$ and $J=4$.}\\

\end{example}

\begin{example} \noindent \textit{Inspired in Example} \rm\ref{example2}\textit{(a) and (c) we finish this work presenting in Table} \rm\ref{tabela1} \textit{values for $m(f(D(f))$ and $J$ for every map germ in Mond's list using the formulas of Theorem} \rm\ref{main theo}. \textit{We also include in Table} \rm\ref{tabela1} \textit{the values of $r_i(D(f))$ and $r_f(D(f))$.}

\begin{table}[!h]
\caption{Quasi-homogeneous map germs in Mond's list \cite{mond6}.}\label{tabela1}
\centering
{\def\arraystretch{1.9}\tabcolsep=4pt 

\begin{tabular}{ c | c | c | M{1cm}| M{1cm} | M{1.4cm} | M{1cm} }

\hline
\rowcolor{lightgray}
Name  &  $f(x,y)=$  &  Quasi-Homogeneous type & {\footnotesize $r_i(D(f))$} & {\footnotesize $r_f(D(f))$}  & {\footnotesize $m(f(D(f)))$}  &  $J$ \\

\hline
Cross-Cap     & $(x,y^2,xy)$ & $(1,2,2;1,1)$ & $0$  & $1$  & $1$ &   $0$    \\
\hline
$S_k$,  $k\geq 1$ odd   & $(x,y^2,y^3+x^{k+1}y)$  & $(1,k+1,\frac{3(k+1)}{2}; 1,\frac{k+1}{2})$ & $2$ & $0$  &    $1$ &  $0$      \\
\hline
$S_k$,  $k\geq 1$ even    & $(x,y^2,y^3+x^{k+1}y)$  & $(2,2k+2,3k+3;2,k+1)$ & $0$  & $1$  &    $1$ & $0$   \\
\hline
$B_k$, $k\geq 3$ odd    & $(x,y^2,y^{2k+1}+x^2y)$  & $(k,2,2k+1;k,1)$ & $2$ & $0$  &    $2$ & $k$      \\
\hline
$B_k$, $k\geq 3$ even   & $(x,y^2,y^{2k+1}+x^2y)$  & $(k,2,2k+1;k,1)$ & $0$ & $2$  &    $2$ & $k$      \\
\hline
$C_k$, $k\geq 3$ odd    & $(x,y^2,xy^3+x^ky)$  & $(1,k-1,\frac{3k-1}{2};1,\frac{k-1}{2})$ & $2$ & $1$  &    $2$  & $2$   \\
\hline   
$C_k$, $k\geq 3$ even   & $(x,y^2,xy^3+x^ky)$ & $(2,2k-2,3k-1;2,k-1)$ & $0$ & $2$  &    $2$  &  $2$      \\
\hline   
$F_4$     & $(x,y^2,y^5+x^3y)$ & $(4,6,15;4,3)$ & $0$ & $1$  & $2$ & $3$      \\
\hline    
$H_k$     & $(x,y^3,y^{3k-1}+xy)$, $k\geq 2$ & $(3k-2,3,3k-1;3k-2,1)$ & $2$ & $0$  &    $3$ &  $2$      \\
\hline
$T_4$     & $(x,y^3+xy,y^4)$ & $(2,3,4;2,1)$ & $2$ & $1$  & $3$  & $3$      \\
\hline  
$P_3$     & $(x,y^3+xy,cy^4+xy^2)^{\ast}$ & $(2,3,4;2,1)$ & $2$ & $1$  &    $3$  & $3$       \\
\hline   

\multicolumn{5}{l}{$\ast \ c\neq 0,1/2,1,3/2$}.

\end{tabular}
}
\end{table}

\end{example}

\begin{flushleft}
\textit{Acknowlegments:} We would like to thank Jawad Snoussi and Guillermo Peñafort-Sanchis for many helpful conversations, suggestions and comments on this work. The author would like to thank CONACyT for the financial support by Fordecyt 265667 and UNAM/DGAPA for support by PAPIIT IN $113817$.
\end{flushleft}

\end{document}